\documentclass[10pt]{amsart}

\usepackage{amsmath,amsthm,amsfonts,latexsym,amssymb,mathrsfs,setspace,enumerate}

\newtheorem{thm}{Theorem}

\newtheorem{lem}[thm]{Lemma}
\newtheorem{prop}[thm]{Proposition}
\newtheorem{cor}[thm]{Corollary}
\numberwithin{thm}{section}
\numberwithin{equation}{section}

\theoremstyle{definition}

\newtheorem*{conj}{Conjecture}

\newcommand{\rat}{\mathbb Q}

\newcommand{\alg}{\overline\rat}
\newcommand{\algt}{\alg^{\times}}

\newcommand{\intg}{\mathbb Z}

\newcommand{\X}{\mathcal X}
\newcommand{\G}{\mathcal G}
\newcommand{\tor}{\mathrm{Tor}}

\title{On the non-Archimedean metric Mahler measure}
\author[P. Fili \and C.L. Samuels]{Paul Fili \and Charles L. Samuels}
\address{Department of Mathematics, University of Texas at Austin, 1 University Station C1200
  Austin, TX 78712}
\email{pfili@math.utexas.edu}
\address{Max-Planck-Institut f\"ur Mathematik, Vivatsgasse 7, 53111 Bonn, Germany}
\email{csamuels@mpim-bonn.mpg.de}

\begin{document}

\begin{abstract}
Recently, Dubickas and Smyth constructed and examined the metric Mahler measure and the metric na\"ive height on
the multiplicative group of algebraic numbers.  We give a non-Archimedean version of the metric Mahler measure,
denoted $M_\infty$, and prove that $M_\infty(\alpha) = 1$ if and only if $\alpha$ is a root of unity.
We further show that $M_\infty$ defines a projective height on $\algt/\tor(\algt)$ as a vector space over $\rat$.
Finally, we demonstrate how to compute $M_\infty(\alpha)$ when $\alpha$ is a surd.
\end{abstract}

\maketitle

\section{Introduction}

Let $K$ be a number field and $v$ a place of $K$ dividing the place $p$ of $\rat$.  Let $K_v$ and $\rat_p$
denote the respective completions.  We write $\|\cdot\|_v$ to denote the unique absolute value on $K_v$ 
extending the $p$-adic absolute value on $\rat_p$ and define 
\begin{equation*}
  |\alpha|_v=\|\alpha \|_v^{[K_v:\rat_p]/[K:\rat]}
\end{equation*}
for all $\alpha\in K$.  Define the {\it Weil height} of $\alpha\in K$ by
\begin{equation*} \label{WeilHeightDef}
  H(\alpha) = \prod_v\max\{1,|\alpha|_v\}
\end{equation*}
where the product is taken over all places $v$ of $K$.  Given this normalization of our absolute values,
the above definition does not depend on $K$, and therefore, $H$ is a well-defined function on $\alg$.
Clearly $H(\alpha)\geq 1$, and by Kronecker's Theorem, we have equality precisely when $\alpha$ is zero 
or a root of unity.

We further define the {\it Mahler measure} of $\alpha\in\alg$ by
\begin{equation*}
  M(\alpha) = H(\alpha)^{\deg \alpha}
\end{equation*}
where $\deg \alpha$ denotes the degree of $\alpha$ over $\rat$.  It is simple to compute
the Mahler measure of $\alpha$ in terms of its minimal polynomial $f_\alpha$ over $\intg$.
If we write
\begin{equation*}
  f_\alpha(z) = A\cdot\prod_{n=1}^N (x-\alpha_n)
\end{equation*}
then, since $H$ is invariant under Galois conjugation over $\rat$, we have that
\begin{equation} \label{MahlerProduct}
	M(\alpha) = \prod_{n=1}^N H(\alpha_n).
\end{equation}

Certainly $M(\alpha) = 1$ if and only if $\alpha$ is a root of unity.  As part of an algorithm for
computing large primes, D.H. Lehmer \cite{Lehmer} asked if there exists a sequence of algebraic numbers, 
none of which are roots of unity, whose Mahler measures tend to $1$.  The smallest Mahler measure that he
found occurs when $\gamma$ is a root of  
\begin{equation*}
  \ell(x) = x^{10}+x^9-x^7-x^6-x^5-x^4-x^3+x+1
\end{equation*}
in which case $M(\gamma) = 1.17\ldots$.  Since Lehmer's famous paper, many algorithms have
been implemented to find numbers of small Mahler measure (see \cite{Moss, MossWeb, MPV}, for instance), and all have failed to produce an 
algebraic number of Mahler measure smaller than $M(\gamma)$.  This led to the conjecture, 
now known as Lehmer's conjecture, that there does not exist such a sequence.

\begin{conj} \label{LehmerConj}
  There exists a constant $c>1$ such that $M(\alpha) \geq c$ whenever $\alpha\in\algt$ is not a
  root of unity.
\end{conj}

Although many special cases have been established (see, for example, \cite{Breusch, Schinzel, Smyth}),
Lehmer's problem remains open in general.  The best known universal lower bound on $M(\alpha)$
is due to Dobrowolski \cite{Dobrowolski}, who proved that 
\begin{equation} \label{Dobrowolski}
  \log M(\alpha) \gg \left(\frac{\log\log \deg \alpha}{\log \deg \alpha}\right)^3
\end{equation}
whenever $\alpha$ is not a root of unity.

Recently, Dubickas and Smyth \cite{DubSmyth2} defined and studied the metric Mahler measure on the multiplicative
group of algebraic numbers.  Specifically, let
\begin{equation*}
	\X(\algt) = \{(\alpha_1,\alpha_2,\ldots): \alpha_n\in \algt,\ \alpha_n = 1\ \mathrm{for\ a.e.}\ n\}.
\end{equation*}
That is, each element $(\alpha_1,\alpha_2,\ldots) \in \X(\algt)$ must have $\alpha_n = 1$ for all but finitely many positive integers $n$.
Also define the map $\tau:\X(\algt)\to \algt$ by $\tau(\alpha_1,\alpha_2,\ldots) = \alpha_1\alpha_2\cdots$ and observe that
$\tau$ is a group homomorphism.  Define the {\it metric Mahler measure} by
\begin{equation} \label{DSMetricMM}
  M_1(\alpha) = \inf\left\{\prod_{n=1}^\infty M(\alpha_n): (\alpha_1,\alpha_2,\ldots) \in \tau^{-1}(\alpha)\right\}
\end{equation}
and note that $M_1(\alpha\beta) \leq M_1(\alpha)M_1(\beta)$ for all $\alpha,\beta\in\algt$.
Using the triangle inequality for the Weil height, one verifies easily that
\begin{equation} \label{MetricBounds}
	M(\alpha) \geq M_1(\alpha) \geq H(\alpha)
\end{equation}
which implies, in particular, that $M_1(\alpha) = 1$ if and only if $\alpha$ is a root of unity.
This means that the map $(\alpha,\beta) \mapsto \log M_1(\alpha\beta^{-1})$ defines a metric on the quotient 
group $\G = \algt/\tor(\algt)$.  
In addition, Dubickas and Smyth prove that $M_1(\alpha) = M(\alpha)$ whenever $\alpha$ is a rational number, 
a Pisot number, a Salem number, or a product of such numbers.
Although it is too technical to include here, they further show how to compute $M_1(\alpha)$ when $\alpha$ is a surd.

In this paper, we examine the following non-Archimedean version of the metric Mahler measure.  Define
\begin{equation} \label{StrongMetricMM}
  M_\infty(\alpha) = \inf\left\{\max_{n\geq 1} M(\alpha_n): (\alpha_1,\alpha_2,\ldots) \in \tau^{-1}(\alpha)\right\}
\end{equation}
and note that $M_\infty(\alpha\beta) \leq \max\{M_\infty(\alpha),M_\infty(\beta)\}$ for all $\alpha,\beta\in \algt$.
Our first goal is to show that $M_\infty(\alpha) = 1$ if and only if $\alpha$ is a root of unity.
This fact is nearly trivial in the case of $M_1$, as it follows easily from inequality \eqref{MetricBounds}.
Although we know that $M(\alpha) \geq M_\infty(\alpha)$, we cannot conclude that $M_\infty(\alpha) \geq H(\alpha)$
because $H$ does not have the strong triangle inequality.  In fact, this inequality is false in general because, for example,
$H(4) = 4$ but $M_\infty(4) \leq 2$.  However, we are able to establish a slightly weaker version.

\begin{thm} \label{StrongMetricBound}
	If $\alpha$ is a non-zero algebraic number and not a root of unity then
	\begin{equation} \label{SMZeroSetEq}
		M_\infty(\alpha) \geq \inf\{H(\gamma):  \gamma\in\rat(\alpha)\ \mathrm{and}\ H(\gamma) > 1\}.
	\end{equation}
\end{thm}

Dobrowolski's Theorem \eqref{Dobrowolski} implies immediately that the right hand side of \eqref{SMZeroSetEq} is strictly greater than $1$.
By Northcott's Theorem \cite{Northcott}, the set $\{\gamma\in\rat(\alpha):  T > H(\gamma)> 1\}$ is finite for every positive real
number $T$.  This means that the infimum in \eqref{SMZeroSetEq} is, in fact, achieved.  Either result is enough to obtain the following corollary. 

\begin{cor} \label{SMZeroSet}
	$M_\infty(\alpha) = 1$ if and only if $\alpha$ is a root of unity.
\end{cor}

In view of Corollary \ref{SMZeroSet}, the map $(\alpha,\beta)\mapsto\log M_\infty(\alpha\beta^{-1})$ defines a metric on $\G$.  Like the
metric Mahler measure $M_1$, $M_\infty$ induces the discrete topology on $\G$ if and only if Lehmer's conjecture is true.

It is important to note that Corollary \ref{SMZeroSet} is trivial under the assumption of Lehmer's conjecture.  
Indeed, if $M_\infty(\alpha) = 1$ and $\alpha$ is not a root of unity, then whenever $\alpha$ is written as a product, some element
of the product must not be a root of unity.  Hence, we obtain a sequence of points $\alpha_n$, none of which are roots of unity, with
$M(\alpha_n)$ tending to $1$ as $n\to\infty$.  Of course, this would contradict the conclusion of Lehmer's conjecture.

We now give some additional basic properties about $M_\infty$.  Let $K$ be a number field and $\alpha\in K$.
For a rational prime $p$, we say that $\alpha$ is a {\it $p$-adic unit} if for every place $v$ dividing $p$, we have that $|\alpha|_v = 1$.
Of course, this definition does not depend on $K$.  Further, it is well-known that $\alpha$ is a $p$-adic unit if and only if $p$ divides neither the first
nor the last coefficient of the minimal polynomial of $\alpha$ over $\intg$.

\begin{thm} \label{BasicProperties}
	If $\alpha\in\algt$ and $r\in\rat^\times$ then $M_\infty(\alpha^r) = M_\infty(\alpha)$.  Moreover, if $p$ is the largest
	prime such that $\alpha$ fails to be a $p$-adic unit then $M_\infty(\alpha) \geq p$.
\end{thm}

In general, there is ambiguity in writing $\alpha^r$ for $r\in\rat$ because there may be many $r$th powers of $\alpha$.  
However, all such powers lie in the same coset of $\tor(\algt)$ in $\algt$.  It is obvious that $M_\infty$ is invariant
under mulitplication by a root of unity so these elements must all have the same value.  Theorem \ref{BasicProperties} further implies that 
$M_\infty$ defines a projective height on $\G$ when it is viewed as a vector space, written multiplicatively, over $\rat$.  This vector space 
is studied extensively in \cite{AV}, in which it is noted that, among other things, the Weil height defines a norm with respect to the usual 
absolute value on $\rat$.

As an example of the second statement of Theorem \ref{BasicProperties}, consider the algebraic number $\gamma = 1 + \sqrt 5$.  It is computed easily
that $\gamma$ has minimial polynomial $x^2 - 2x - 4\in \intg[x]$ so that $\gamma$ fails to be a $2$-adic unit but is a $p$-adic unit for all primes $p > 2$.
In this case, Theorem \ref{BasicProperties} yields the bound $M_\infty(\gamma) \geq 2$.  As another basic example, if $\alpha$ is rational then $M_\infty(\alpha)$
is bounded below by the largest prime that divides its numerator or denominator.  In fact, we may apply 
Theorem \ref{BasicProperties} to compute precisely the value of the strong metric Mahler measure at any surd.

\begin{cor} \label{Surd}
  If $\alpha$ is rational and $d$ is a positive integer then $M_\infty(\alpha^{1/d})$ equals the largest prime dividing
  the numerator or denominator of $\alpha$.
\end{cor}

It is worth noting that Corollary \ref{Surd} identifies a large class of cases of equality in the second statement of Theorem \ref{BasicProperties}.
Indeed, the primes dividing the numerator and denominator of $\alpha$ are the same as the primes such that $\alpha^{1/d}$ fails to be
a $p$-adic unit.

\section{Heights on Abelian groups} \label{GroupHeight}

The method used to construct \eqref{DSMetricMM} and \eqref{StrongMetricMM} is applicable on any abelian group with a function
satisfying only a few simple properties.  Although we cannot hope to prove anything particularly deep in such a general setting,
it is worth exploring the basic facts before we prove our main results.

Let $G$ be an abelian group written multiplicatively.  We say that $\rho:G\to[1,\infty)$ is a {\it height on $G$} if
the conditions
\begin{enumerate}[(i)]
  \item $\rho(1) = 1$
  \item $\rho(\alpha) = \rho(\alpha^{-1})$
\end{enumerate}
are satisfied.  We define the {\it zero set of $\rho$} to be 
\begin{equation*}
  Z(\rho) = \{\alpha\in G: \rho(\alpha) = 1\}.
\end{equation*}
We further say that $\rho$ is a {\it metric height on $G$} if we have that
$$\rho(\alpha\beta)\leq \rho(\alpha)\rho(\beta)$$ for all $\alpha,\beta\in G$.  If $\rho$ satisfies the stronger condition
that $$\rho(\alpha\beta)\leq\max\{\rho(\alpha),\rho(\beta)\}$$ for all $\alpha,\beta\in G$ then we say that $\rho$ is a 
{\it strong} (or {\it non-Archimedean}) {\it metric height on $G$}.  If $\sigma$ is another height on $G$ then we write $\sigma\leq\rho$ if
$\sigma(\alpha) \leq \rho(\alpha)$ for all $\alpha\in G$.  This yields a partial ordering of the set of all 
heights on $G$.

As we noted in the introduction, 
Dubickas and Smyth \cite{DubSmythLength, DubSmyth, DubSmyth2} studied several heights and metric heights on the group of algebraic 
numbers $\algt$.  More specifically, they defined and studied the metric heights associated to the Mahler measure, the na\"ive height, and the length. 
Our first proposition generalizes several facts noted by Dubickas and Smyth regarding metric heights.
The proof is only trivially different from several remarks made in \cite{DubSmyth} and \cite{DubSmyth2},
however, we include it here for the purposes of completeness.

\begin{prop}\label{MetricHeightProperties}
  Suppose that $\rho$ is a metric height on the abelian group $G$.  Then
  \begin{enumerate}[(i)]
  \item $Z(\rho)$ is a subgroup of $G$.
  \item\label{Defined}$\rho(\alpha) = \rho(\zeta\alpha)$ for all $\alpha\in G$ and $\zeta\in Z(\rho)$.  
    That is, $\rho$ is well-defined on $G/Z(\rho)$.
  \item\label{Metric}The map $(\alpha,\beta) \mapsto \log \rho(\alpha\beta^{-1})$ defines a metric on $G/Z(\rho)$.
  \end{enumerate}
\end{prop}
\begin{proof}
  If $\rho(\alpha) = \rho(\beta) = 1$ then we know that $\rho(\alpha\beta) \leq \rho(\alpha)\rho(\beta) = 1$.
  By definition of height we conclude that $Z(\rho)$ is indeed a subgroup of $G$.
  If $\zeta\in Z(\rho)$ then we have that
  \begin{equation*}
    \rho(\alpha) = \rho(\zeta^{-1}\zeta\alpha) \leq \rho(\zeta^{-1})\rho(\zeta\alpha) = \rho(\zeta\alpha)
    \leq \rho(\zeta)\rho(\alpha) = \rho(\alpha)
  \end{equation*}
  which establishes that $\rho(\alpha) = \rho(\zeta\alpha)$.
  The final statement of the proposition follows from the triangle inequality.
\end{proof}

Of course, Proposition \ref{MetricHeightProperties} justifies our use of the word metric in the definition of 
metric height: although $\rho$ does not necessarily define a metric on $G$, it is indeed a well defined metric on the 
quotient $G/Z(\rho)$.  Thus it is important to identify the subgroup $Z(\rho)$ if we hope to fully understand a 
metric height $\rho$.

If we are given a height $\rho$ on $G$ it is possible to construct both a natural metric height and a natural
strong metric height from $\rho$.  Let
\begin{equation*}
  \X(G) = \{(\alpha_1,\alpha_2,\ldots):\alpha_n\in G,\ \alpha_n=1\ \mathrm{for\ a.e.}\ n\}.
\end{equation*}
Further, define the map $\tau:\X(G)\to G$ by $\tau(\alpha_1,\alpha_2,\ldots) = \alpha_1\alpha_2\cdots$
and note that $\tau$ is a surjective group homomorphism.  As is done in \cite{DubSmyth} and \cite{DubSmyth2}
using the Mahler measure and na\"ive height, we define
\begin{equation}\label{MetricHeightDef}
  \rho_1(\alpha) = \inf\left\{\prod_{n=1}^\infty \rho(\alpha_n):(\alpha_1,\alpha_2,\ldots) \in \tau^{-1}(\alpha)\right\}
\end{equation}
and note that the map $\rho\mapsto\rho_1$ preseves the partial ordering of heights on $G$.  In other words,
if $\rho$ and $\sigma$ are heights on $G$ with $\sigma\leq\rho$ then $\sigma_1\leq \rho_1$.
Now we establish a modification of the results of \cite{DubSmyth} and \cite{DubSmyth2}.

\begin{thm} \label{MetricConstruction}
  If $\rho$ is a height on $G$ then the following hold.
  \begin{enumerate}[(i)]
  \item\label{MetricHeightConversion} $\rho_1$ is a metric height on $G$ with $\rho_1\leq\rho$.
  \item\label{BestMetricHeight} If $\sigma$ is a metric height with $\sigma\leq\rho$ then 
    $\sigma\leq \rho_1$.
  \item\label{NoChangeMetric} $\rho = \rho_1$ if and only if $\rho$ is metric height.
  \item\label{NoDoubleMetric} $(\rho_1)_1 = \rho_1$.
  \end{enumerate}
\end{thm}
\begin{proof}
  It is obvious that $\rho_1(\alpha)\geq 1$ for all $\alpha\in G$ and that $\rho_1(1) = 1$.  Since
  $\alpha\mapsto\alpha^{-1}$ is an automorphism of $G$ and $\rho(\alpha) = \rho(\alpha^{-1})$ for all
  $\alpha\in G$, it is also clear that $\rho_1(\alpha) = \rho_1(\alpha^{-1})$.
  Since $\tau(\alpha,1,1,\ldots) = \alpha$, we have that $\rho_1\leq\rho$ as well.  To prove the triangle inequality
  for $\rho_1$ we observe that
  \begin{align*}
    \rho_1(\alpha\beta) & 
    = \inf\left\{\prod_{n=1}^\infty \rho(\alpha_n): (\alpha_1,\alpha_2,\ldots) \in \tau^{-1}(\alpha\beta)\right\} \\
    & \leq \inf\left\{\prod_{n=1}^\infty \rho(\alpha_n)\prod_{m=1}^\infty \rho(\beta_m): 
      (\alpha_1,\alpha_2,\ldots) \in \tau^{-1}(\alpha), (\beta_1,\beta_2,\ldots) \in \tau^{-1}(\beta)\right\} \\
    %& = \inf\left\{\prod_{n=1}^\infty \rho(\alpha_n): \tau(\alpha_1,\alpha_2,\ldots) = \alpha\right\}\cdot
    %\inf\left\{\prod_{m=1}^\infty \rho(\beta_m): \tau(\beta_1,\beta_2,\ldots) = \beta\right\} \\
    & = \rho_1(\alpha)\rho_1(\beta)
  \end{align*}
  which establishes \eqref{MetricHeightConversion}.

  To prove \eqref{BestMetricHeight}, we observe that
  \begin{align*}
    \rho_1(\alpha) & = \inf\left\{\prod_{n=1}^\infty \rho(\alpha_n): (\alpha_1,\alpha_2,\ldots) \in \tau^{-1}(\alpha)\right\}\\
    & \geq \inf\left\{\prod_{n=1}^\infty \sigma(\alpha_n): (\alpha_1,\alpha_2,\ldots) \in \tau^{-1}(\alpha)\right\}\\
    & \geq \sigma(\alpha)
  \end{align*}
  where the last inequality follows from the triangle inequality for $\sigma$.

   Obviously if $\rho = \rho_1$ then $\rho$ is a metric height.   To prove the converse, we asssume that
   $\rho$ is a metric height.  Hence,  statement \eqref{BestMetricHeight} implies that
   $\rho\leq\rho_1\leq\rho$ which yields our result.  The final statement follows immediately since
   $\rho_1$ is itself a height.
\end{proof}

Indeed, Theorem \ref{MetricConstruction} indicates that the definition \eqref{MetricHeightDef} is a 
natural way of constructing a metric height out of an oridinary height.  Not only do we obtain a metric height,
but we obtain the largest metric height that is less than or equal to $\rho$.  Furthermore, we need not attempt this 
construction with a height that is already known to be metric.  For example, the Weil height
$H$ on $\algt$ is already a metric height so that applying \eqref{MetricHeightDef} yields the Weil height again.
The Mahler measure of an algebraic number $\alpha$, however, does not have the triangle inequality, so this leads to the
non-trivial construction studied in \cite{DubSmyth2}.

We now turn our attention to a non-Archimedean version of \eqref{MetricHeightDef}.  
Once again, we assume that $\rho$ is a height on $G$ and define
\begin{equation} \label{StrongMetricHeightDef}
  \rho_\infty(\alpha) = \inf\left\{\max_{n\geq 1}\left\{\rho(\alpha_n)\right\}: 
    (\alpha_1,\alpha_2,\ldots) \in\tau^{-1}(\alpha)\right\}
\end{equation}
so that the product in \eqref{MetricHeightDef} is replaced with a maximum.  As in the construction of $\rho_1$,
we observe that the strong metric contruction preserves the partial ordering of heights on $G$.  
We further note an analogue of Theorem \ref{MetricConstruction} for $\rho_\infty$.

\begin{thm} \label{StrongMetricConstruction}
  If $\rho$ is a height on $G$ then the following hold.
  \begin{enumerate}[(i)]
  \item\label{SMetricHeightConversion} $\rho_\infty$ is a strong metric height on $G$ with $\rho_\infty\leq\rho_1$.
  \item\label{SBestMetricHeight} If $\sigma$ is a strong metric height with $\sigma\leq\rho$ then 
    $\sigma\leq \rho_\infty$.
  \item\label{SNoChangeMetric} $\rho = \rho_\infty$ if and only if $\rho$ is a strong metric height.
  \item\label{SNoDoubleMetric} $\rho_\infty = (\rho_1)_\infty = (\rho_\infty)_1 = (\rho_\infty)_\infty$.
  \end{enumerate}
\end{thm}
\begin{proof}
  The proofs of statements \eqref{SMetricHeightConversion}, \eqref{SBestMetricHeight} and 
  \eqref{SNoChangeMetric} are nearly identical to proofs of the analogous statements in Theorem 
  \ref{MetricConstruction} so we do not include them here.  To verify \eqref{SNoDoubleMetric} we note that
  \begin{equation*}
    \rho_\infty \leq \rho_1 \leq \rho.
  \end{equation*}
  Since these inequalities are preseved by taking the strong metric height of each component, we obtain
  \begin{equation*}
    (\rho_\infty)_\infty \leq (\rho_1)_\infty \leq \rho_\infty.
  \end{equation*}
  But it is clear from the \eqref{SNoChangeMetric} that $(\rho_\infty)_\infty = \rho_\infty$ so that
  \begin{equation*}
    (\rho_\infty)_\infty = (\rho_1)_\infty = \rho_\infty.
  \end{equation*}
  Finally, we note that $\rho_\infty$ is certainly a metric height so that $(\rho_\infty)_1 = \rho_\infty$ by
  Theorem \ref{MetricConstruction} \eqref{NoChangeMetric}.
\end{proof}

Theorem \ref{StrongMetricConstruction} implies that $\rho_\infty$ is indeed a metric height
as well so that we may apply Proposition \ref{MetricHeightProperties} to it.  The metric induced by $\rho_\infty$
on $G/Z(\rho_\infty)$ is non-Archimedean, so every open or closed ball centered at $1$ is a subgroup of 
$G/Z(\rho_\infty)$.  Furthermore, for any $r\geq 1$, we set
\begin{equation*}
  B_r = \{\alpha\in G: \rho_\infty(\alpha) < r \}
\end{equation*}
and let $S_r$ be the subgroup of $G$ generated by the set $\{\alpha\in G: \rho(\alpha) < r\}$.
It is clear that
\begin{equation*}
  S_r = \{\tau(\alpha_1,\alpha_2,\ldots): (\alpha_1,\alpha_2,\ldots)\in\X(G)\ \mathrm{and}\
  \rho(\alpha_n) < r\ \mathrm{for\ all}\ n\}.
\end{equation*}
If $\alpha\in S_r$ then $\alpha = \prod_{n=1}^\infty\alpha_n$ where $\rho(\alpha_n) < r$ for all $n$.
Hence, 
\begin{equation*}
  \rho_\infty(\alpha) \leq \max_{n\geq 1}\{\rho(\alpha_n)\} < r
\end{equation*}
so that $\alpha\in B_r$.  To establish the opposite containment, note that if $\alpha\in B_r$ then 
$\rho_\infty(\alpha) < r$.  Therefore, by definition of 
$\rho_\infty$ there exists $(\alpha_1,\alpha_2,\ldots) \in \tau^{-1}(\alpha)$ such that
$\max_{n\geq 1}\{\rho(\alpha_n)\} < r$.  It follows that $\alpha\in S_r$ and we have shown that
\begin{equation} \label{OpenBall}
  B_r = S_r.
\end{equation}
It is worth noting that there is no analog of \eqref{OpenBall} for closed balls unless the infimum in
$\rho_\infty$ is always achieved on the boundary of the ball.  
If $\rho_\infty(\alpha) = r$ then we may simply conclude that for every $\varepsilon > 0$ there
exists $(\alpha_1,\alpha_2,\ldots) \in \tau^{-1}(\alpha)$ such that
$\max_{n\geq 1}\{\rho(\alpha_n)\} < r + \varepsilon$.  However, one cannot conclude that 
$\max_{n\geq 1}\{\rho(\alpha_n)\} \leq r$.

As an example of \eqref{StrongMetricHeightDef},  we note that the Weil height $H$ does not already have 
the strong triangle inequality.  Therefore, we may find it interesting to apply \eqref{StrongMetricHeightDef} 
to it.  However, we quickly realize that if $\alpha\in\algt$ then we may write $\alpha = (\alpha^{1/n})^n$
so that $H_\infty(\alpha)\leq H(\alpha^{1/n}) = H(\alpha)^{1/n}$.  But, $H(\alpha)^{1/n}$ tends to $1$ as $n$
tends to $\infty$ implying that $H_\infty$ is trivial.  Of course, Corollary \ref{SMZeroSet} establishes that
$M_\infty$ is non-trivial.

\section{Proofs of our main results} \label{Proofs}

Before we prove Theorem \ref{StrongMetricBound} we recall the relevant definitions and notation.  Suppose that $K/F$ is any finite Galois 
extension of fields and let $G = \mathrm{Aut}(K/F)$.  Recall that the {\it norm from $K$ to $F$} is the map $\mathrm{Norm}_{K/F}:K\to F$ 
defined by
\begin{equation*}
	\mathrm{Norm}_{K/F}(\alpha) = \prod_{\sigma\in G}\sigma(\alpha).
\end{equation*}
It is obvious that right hand side is invariant under Galois conjugation by an element of $G$ so that $\mathrm{Norm}_{K/F}(\alpha)$ does indeed
belong to $F$.  Of course, if $\alpha\in F$ then $\mathrm{Norm}_{K/F}(\alpha) = \alpha^{[K:F]}$.  Furthermore, 
$\mathrm{Norm}_{K/F}(\alpha\beta) = \mathrm{Norm}_{K/F}(\alpha)\mathrm{Norm}_{K/F}(\beta)$ so that the norm is a homomorphism from $K^\times$ to $F^\times$.

\bigskip

\noindent{\it Proof of Theorem \ref{StrongMetricBound}}.
Assume $\alpha\in\algt\setminus\tor(\algt)$ and let $\varepsilon > 0$.  Further suppose that $(\alpha_1,\alpha_2,\ldots)\in \X(\algt)$ is such that
$\alpha= \alpha_1\alpha_2\cdots$ and
\begin{equation} \label{Representation}
	\max\{M(\alpha_1),M(\alpha_2),\ldots\} \leq M_\infty(\alpha) + \varepsilon.
\end{equation}
Let $F = \rat(\alpha)$ and assume that $K$ is a Galois extension of $F$ containing each 
element $\alpha_n$.  Since $\alpha_n = 1$ for almost every $n$, the Galois group $G = \mathrm{Aut}(K/F)$ is finite.

First assume that $\mathrm{Norm}_{K/F}(\alpha_n)$ is a root of unity for all $n$.  Then we have that
\begin{equation*}
	\alpha^{[K:F]} = \mathrm{Norm}_{K/F}(\alpha) = \prod_{n=1}^\infty \mathrm{Norm}_{K/F}(\alpha_n)
\end{equation*}
since the norm is a multiplicative homomorphism.  Therefore, $\alpha$ is a root of unity which is a
contradiction.

Now we may assume that there exists $\beta \in \{\alpha_1,\alpha_2,\ldots\}$ such that $\mathrm{Norm}_{K/F}(\beta)$
is not a root of unity.  For simplicity, we let $H = \mathrm{Aut}(K/F(\beta))$ and let $S$ be a complete set of coset 
representatives of $H$ in $G$.  Also, assume that $\beta_1,\ldots,\beta_M$ are the conjugates of $\beta$ over $F$ 
and let $\hat\beta = \beta_1\cdots\beta_M$.  We obtain that
\begin{equation*}
	\mathrm{Norm}_{K/F}(\beta) = \prod_{\sigma\in G}\sigma(\beta) = \prod_{\sigma\in S}\sigma(\beta)^{|H|} = \hat\beta^{|H|}
\end{equation*}
so that $\hat\beta$ must not be a root of unity.
	
Using \eqref{MahlerProduct}, we know that $M(\beta)$ is the product of the heights of the conjugates of $\beta$ over $\rat$.  
So the product of the heights of its conjugates over $F$ is potentially smaller.  Then using the triangle inequality for the Weil height, we find that
\begin{equation*}
	M(\beta) \geq \prod_{m=1}^M H(\beta_m) \geq H(\hat\beta)
\end{equation*}
and it follows that
\begin{equation} \label{EpsilonBound}
	M_\infty(\alpha) + \varepsilon \geq H(\hat\beta) \geq \inf\{\gamma\in\rat(\alpha):  H(\gamma)> 1\}.
\end{equation}
Since the right hand side of \eqref{EpsilonBound} does not depend on $\varepsilon$, we may let $\varepsilon$ tend to zero to obtain the desired result.\qed
\bigskip

\noindent {\it Proof of Corollary \ref{SMZeroSet}}.
If $\alpha$ is a root of unity, then we have $M_\infty(\alpha) \leq M(\alpha) = 1$ so that $M_\infty(\alpha) = 1$.  If
$\alpha$ is not a root of unity, then Theorem \ref{StrongMetricBound} gives
\begin{equation*}
	M_\infty(\alpha) \geq \inf\{\gamma\in\rat(\alpha):  H(\gamma)> 1\}.
\end{equation*}
However, Dobrowolski's Theorem implies that the right hand side is stricly greater than $1$ which establishes the corollary.\qed
\bigskip

The proof of Theorem \ref{BasicProperties} will require a technical lemma.

\begin{lem} \label{Degree}
  Let $\alpha\in\algt$ and define $d(\alpha) = \min\{\deg(\zeta\alpha):\zeta\in\tor(\algt)\}$. 
  We have that $d(\alpha)\leq r d(\alpha^r)$ for all positive integers $r$ and
  \begin{equation*}
    M_\infty(\alpha) = \inf\left\{\max_{n\geq 1}\left\{H(\alpha_n)^{d(\alpha_n)}\right\}: (\alpha_1,\alpha_2,\ldots) 
      \in \tau^{-1}(\alpha)\right\}.
  \end{equation*}
\end{lem}
\begin{proof}
  To prove the first statement, let $f(x)$ denote the minimal polynomial of $\alpha^r$ over $\intg$.  Hence,
  the polynomial $f(x^r)$ vanishes at $\alpha$ and has degree $r\deg \alpha^r$.  It follows that
  $\deg\alpha\leq r\deg \alpha^r$.  But then
  \begin{align*}
    r d(\alpha^r) & = r\cdot\inf\{\deg(\zeta\alpha^r):\zeta\in\tor(\algt)\} \\
    & = \inf\{r\deg(\zeta^r\alpha^r):\zeta\in\tor(\algt)\} \\
    & \geq \inf\{\deg(\zeta\alpha):\zeta\in\tor(\algt)\} = d(\alpha).
  \end{align*}
  
  To prove the second statement, we first observe that $\deg(\alpha)\geq d(\alpha)$ so that
  \begin{equation} \label{GEQ}
    M_\infty(\alpha) \geq  \inf\left\{\max_{n\geq 1}\left\{H(\alpha_n)^{d(\alpha_n)}\right\}: 
      (\alpha_1,\alpha_2,\ldots)  \in \tau^{-1}(\alpha)\right\}.
  \end{equation}
  Now assume that $\tau(\alpha_1,\alpha_2,\ldots)  = \alpha$ so that $\alpha = \prod_{n=1}^\infty\alpha_n$.
  For each $n$ we select $\zeta_n$ such that $\deg(\zeta_n\alpha_n) = d(\alpha_n)$.  We have that
  \begin{equation*}
    \alpha = \prod_{n=1}^\infty (\alpha_n\zeta_n)\zeta_n^{-1},
  \end{equation*}
  and therefore,
  \begin{align*}
    M_\infty(\alpha) & \leq \max_{n\geq 1}\left\{
      \max\left\{H(\alpha_n\zeta_n)^{\deg(\alpha_n\zeta_n)},H(\zeta_n^{-1})^{\deg(\zeta_n^{-1})}\right\}\right\} \\
    & = \max_{n\geq 1}\left\{H(\alpha_n)^{d(\alpha_n)}\right\}
  \end{align*}
  The result follows by taking the infimum of both sides over all
  $(\alpha_1,\alpha_2,\ldots)  \in\tau^{-1}(\alpha)$.
\end{proof}

\begin{proof}[Proof of Theorem \ref{BasicProperties}]
We first prove that $M_\infty(\alpha^r) = M_\infty(\alpha)$ for all positive integers $r$.  The strong
triangle inequality implies immediately that
\begin{equation*}
  M_\infty(\alpha^r) \leq M_\infty(\alpha)
\end{equation*}
so we must prove the opposite inequality.  By Lemma \ref{Degree} we have that
\begin{equation*}
  M_\infty(\alpha) = \inf\left\{\max_{n\geq 1}\left\{H(\alpha_n)^{d(\alpha_n)}\right\}: (\alpha_1,\alpha_2,\ldots) 
    \in \tau^{-1}(\alpha)\right\}.
\end{equation*}
Each term $H(\alpha)^{d(\alpha)}$ is well-defined on the quotient group $\G=\algt/\tor(\algt)$.  Hence, we may
instead take the infimum over all $(\alpha_1,\alpha_2,\ldots)\in \X(\G)$ with 
$\tau(\alpha_1,\alpha_2,\ldots) = \alpha$ and we obtain the same value.  Applying both statements of
Lemma \ref{Degree} we obtain that
\begin{align*}
  M_\infty(\alpha) & = \inf\left\{\max_{n\geq 1}\left\{H(\alpha_n)^{d(\alpha_n)}\right\}: 
    \alpha = \prod_{n=1}^\infty\alpha_n\right\} \\
  & \leq \inf\left\{\max_{n\geq 1}\left\{H(\alpha_n)^{rd(\alpha_n^r)}\right\}: 
    \alpha = \prod_{n=1}^\infty\alpha_n\right\} \\
  & = \inf\left\{\max_{n\geq 1}\left\{H(\alpha_n^r)^{d(\alpha_n^r)}\right\}: 
    \alpha = \prod_{n=1}^\infty\alpha_n\right\}
\end{align*}
for all $\alpha\in \G$ and all positive integers $r$.

Now define $g_r:\G\to \G$ by $g_r(\alpha) = \alpha^r$ and note that $g_r$ is an automorphism of $\G$.  
We have shown that
\begin{equation} \label{AutoSwitch}
  M_\infty(\alpha) \leq \inf\left\{\max_{n\geq 1}\left\{H(g_r(\alpha_n))^{d(g_r(\alpha_n))}\right\}: 
    \alpha = \prod_{n=1}^\infty\alpha_n\right\}.
\end{equation}
Since $g_r^{-1}$ is also an automorphism, we may take the infimum on the right hand side of \eqref{AutoSwitch}
over all $(\alpha_1,\alpha_2,\ldots)\in \X(\G)$ such that $\alpha = \prod_{n=1}^\infty g_r^{-1}(\alpha_n)$.
We conclude that
\begin{align*}
  M_\infty(\alpha) & \leq \inf\left\{\max_{n\geq 1}\left\{
      H(g_r(g_r^{-1}(\alpha_n)))^{d(g_r(g_r^{-1}(\alpha_n)))}\right\}: 
    \alpha = \prod_{n=1}^\infty g_r^{-1}(\alpha_n)\right\} \\
  & = \inf\left\{\max_{n\geq 1}\left\{H(\alpha_n)^{d(\alpha_n)}\right\}: 
    g_r(\alpha) = \prod_{n=1}^\infty\alpha_n\right\} \\
  & = \inf\left\{\max_{n\geq 1}\left\{H(\alpha_n)^{d(\alpha_n)}\right\}: 
    \alpha^r = \prod_{n=1}^\infty\alpha_n\right\} \\
  & = M_\infty(\alpha^r)
\end{align*}
which completes the proof of the first statement when $r$ is a positive integer.  If $r<0$ is an integer then
\begin{equation*}
  M_\infty(\alpha^r) = M_\infty((\alpha^{-1})^{-r}) = M_\infty(\alpha^{-1}) = M_\infty(\alpha).
\end{equation*}
If we have $r/s\in\rat$ then 
\begin{equation*}
  M_\infty(\alpha^{r/s}) = M_\infty((\alpha^{r/s})^s) = M_\infty(\alpha^r) = M_\infty(\alpha).
\end{equation*}

To prove the second statement we note that if $\alpha$ fails to be a $p$-adic unit then $M(\alpha) \geq p$.  To see this,
let $K = \rat(\alpha)$ and let $v$ be a place of $K$ such that $|\alpha|_v \ne 1$.  Since $M(\alpha) = M(\alpha^{-1})$
me may assume without loss of generality that $|\alpha|_v > 1$.  Further, write 
\begin{equation*}
	O_v = \{z\in K_v: |z|_v \leq 1\}\quad \mathrm{and} \quad M_v = \{z\in K_v: |z|_v < 1\}
\end{equation*}
for the ring of $v$-adic integers in $K_v$ and its unique maximal ideal, respectively.  Let $\pi_v$ be a generator
of $M_v$ so that whenever $|z|_v > 1$ we have that $|z|_v \geq |\pi_v|_v^{-1}$.  It is also well-known that 
\begin{equation} \label{UniformParam}
	|\pi_v|_v^{-[K:\rat]} = p^{f_v}
\end{equation}
where $p^{f_v}$ is the cardinality of the residue field $O_v/M_v$.  We now notice that
\begin{align*}
	M(\alpha) & = \prod_w \max\{1,|\alpha|_w\}^{\deg \alpha} \\
		& \geq |\alpha|_v^{\deg \alpha} \\
		& \geq |\pi_v|_v^{-[K:\rat]} \\
		& = p^{f_v} \geq p.
\end{align*}

Now let $\varepsilon > 0$ and assume that $(\alpha_1,\alpha_2,\ldots)\in \X(\algt)$ is such that 
$\alpha = \alpha_1\alpha_2\cdots$ and 
\begin{equation*}
  M_\infty(\alpha) \geq \max\{M(\alpha_n): n\geq 1\} - \varepsilon.
\end{equation*}
Further assume that $p$ is the largest prime such that $\alpha$ is not a $p$-adic unit.  By our earlier remarks there exists $n$ such that $\alpha_n$ fails to be a $p$-adic 
unit, and thus $M(\alpha_n) \geq p$ and
\begin{equation*}
  M_\infty(\alpha) \geq \max\{M(\alpha_n): n\geq 1\} - \varepsilon \geq p -\varepsilon.
\end{equation*}
The result follows by letting $\varepsilon$ tend to zero.
\end{proof}

\begin{proof}[Proof of Corollary \ref{Surd}]
Since $\alpha\in\rat$ we may write
\begin{equation*}
  \alpha = \prod_{n=1}^N q_n^{r_n}
\end{equation*}
where $q_n$ are rational primes and $r_n$ are non-zero integers.  Assume that $p$ is the largest of the
primes $q_n$.   Then the strong triangle inequality for $M_\infty$ and Theorem \ref{BasicProperties} 
imply that
\begin{equation*}
  M_\infty(\alpha) \leq \max_{1\leq n\leq N}M_\infty(q_n^{r_n}) = \max_{1\leq n\leq N}M_\infty(q_n) 
  \leq \max_{1\leq n\leq N} q_n = p.
\end{equation*}
By the second statement of Theorem \ref{BasicProperties} we also know that $M_\infty(\alpha) \geq p$
so that $M_\infty(\alpha) = p$.  Then applying the first statement again we obtain that
\begin{equation*}
  M_\infty(\alpha^{1/d}) = M_\infty(\alpha) = p.\qedhere
\end{equation*}
\end{proof}

\end{document}